\documentclass[12pt]{amsart}
\usepackage{color,graphicx,array, amssymb, amscd,slashed}
\makeatletter 
\newif\if@noproofreadmark \@noproofreadmarkfalse
\def\noproofreadmark{\@noproofreadmarktrue}
\def\addspan#1{%
 \if@noproofreadmark%
   #1%
 \else%
  \textcolor{red}{#1}%
 \fi%
}
\makeatother 
\newtheorem{Theorem}{Theorem}[section]

\newtheorem{Proposition}[Theorem]{Proposition}
\newtheorem{Corollary}[Theorem]{Corollary}
\theoremstyle{definition}
\newtheorem{Definition}{Definition}
\theoremstyle{remark}
\newtheorem{Remark}[Theorem]{Remark} 
\numberwithin{equation}{section}
\newcommand{\R}{\mathbb R}
\newcommand{\E}{\mathbb E}
\newcommand{\C}{\mathbb C}
\newcommand{\Z}{\mathbb Z}
\newcommand{\D}{\mathbb D}

\newcommand{\ps}{\sc \small PS}
\newcommand{\SU}{{\rm SU}_2}
\newcommand{\su}{\mathfrak {su}_2}
\newcommand{\LSU}{\Lambda {\rm SU}_2}
\newcommand{\lsu}{\Lambda {\mathfrak {su}}_2}
\newcommand{\SL}{{\rm SL}_2 \mathbb C}

\newcommand{\LSUP}{\Lambda^+ {\rm SU}_2}
\newcommand{\LSUN}{\Lambda^- {\rm SU}_2}
\newcommand{\LSUPN}{\Lambda^+_{*} {\rm SU}_2}
\newcommand{\LSUNN}{\Lambda^-_{*} {\rm SU}_2}
\newcommand{\LSUPM}{\Lambda^{\pm} {\rm SU}_2}

\newcommand{\id}{\operatorname{Id}}
\newcommand{\ad}{\operatorname{Ad}}
\newcommand{\di}{\operatorname{diag}}
\newcommand{\offdi}{\operatorname{offdiag}}

\newcommand{\etn}{\eta_n}
\newcommand{\etm}{\eta_m}
\newcommand{\Fn}{F_{n}}
\newcommand{\Gm}{G_{m}}
\renewcommand{\l}{\lambda}

\newcommand{\eq}[1]{\begin{align*}#1 \end{align*}}
\begin{document}
\title{Nonlinear d'Alembert formula for 
 discrete pseudospherical surfaces}
\author[S.-P.~Kobayashi]{Shimpei Kobayashi}
 \address{Department of Mathematics, Hokkaido University, Sapporo, 060-0810, Japan}
 \email{shimpei@math.sci.hokudai.ac.jp}
 \thanks{The author is partially supported by Kakenhi 26400059}
\date{\today}
\pagestyle{plain}
\begin{abstract} 
 On the basis of loop group decompositions (Birkhoff decompositions), we give 
 a discrete version of the nonlinear d'Alembert formula, 
 a method of separation of variables of difference equations, for 
 discrete constant negative Gauss curvature (pseudospherical) surfaces in 
 Euclidean three space. We also compute two examples by this formula in detail.
\end{abstract}
\maketitle
%%%%%%    TEXT START    %%%%%%%%
\section*{Introduction}
 The study of constant negative Gauss curvature surfaces 
 (pseudospherical surfaces, {\ps} surfaces in this paper) 
 in Euclidean three space $\E^3$ was important in classical 
 differential geometry, see for example \cite{Bi}.
 On the one hand the method of integrable systems 
 in theory of differential equations, which started from 1960s, gave
 a new aspect for the study of {\ps} surfaces: 
 The structure equation for {\ps} surfaces 
 is a famous integrable system {\it sine-Gordon} equation, that is, 
 $u_{xy}- \sin u =0$ under the {\it asymptotic Chebyshev coordinates} $(x, y) \in \R^2$
 for {\ps} surfaces.

 Moreover, {\ps} surfaces can be characterized as the existence of 
 a moving frame with an additional parameter $\l \in \R^{\times}$, the so-called 
 {\it spectral parameter}, and the compatibility condition of the moving frame 
 becomes the sine-Gordon equation, that is, {\ps} surfaces can be characterized by 
 the existence of a family of flat connections 
 on a trivial bundle $\R^2 \times \SU$. Such moving frame is called the 
 {\it extended frame}, and it is known that derivative of the extended frame 
 with respect to the spectral parameter reproduces the original {\ps} surface.
 It is fundamental to see that the extended frame 
 takes values in the loop group of $\SU$ and thus the loop group method 
 can be applied for {\ps} surfaces. 

 In \cite{Krichever, Toda, BS}, it was shown that loop group decompositions 
 (Birkhoff decompositions, see Theorem \ref{thm:Birkhoff}) 
 of the extended frame $F$ of a {\ps} surface 
 induced a pair of $1$-forms $(\xi_+, \xi_-)$, 
 that is, $F=F_+ F_-=G_- G_+$ with $\xi_+ = F_+^{-1} d F_+$ and 
 $\xi_- = G_-^{-1} d G_-$. Then it was proved that 
 $\xi_+$ and $\xi_-$ depended only on $x$ and $y$, respectively.
 Conversely it was shown that solving the pair of ordinary differential
 equations $d F_+ = F_+ \xi_+$ and $d G_+ = G_+ \xi_-$ and using 
 the loop group decomposition, the extended frame could be recovered.
 This construction is called the {\it nonlinear d'Alembert formula} for 
 {\ps} surfaces.

 A discrete analogue of smooth {\ps} surfaces was defined in \cite{BP}: 
 First of all discrete surfaces are maps from quadrilateral lattice
 $(n, m) \in \Z^2$ into $\E^3$. Then discrete {\ps} surfaces are 
 defined in terms of some geometric properties which are analogous 
 to the case of smooth {\ps} surfaces. 
 It is essential to see that the discrete {\ps} surfaces 
 have a discrete extended frame which is an exact analogue of the smooth 
 extended frame of a {\ps} surface. Therefore it is natural to expect 
 that the loop group method can be applied. In fact in \cite{PW}, the 
 discrete {\ps} surfaces were obtained by a loop group action, and 
 in \cite{Hoffmann}, several examples were discussed with taking  into 
 account a loop group structure.

 In this paper, we show that there exists a discrete version of 
 nonlinear d'Alembert formula for discrete {\ps} surfaces:
 Decomposing the discrete extended frame $F$ by loop group decompositions 
 (Birkhoff decompositions), 
 that is $F = F_+ F_-$ and $F =G_- G_+$,
 the discrete extended frame $F$ induces a pair of 
 discrete potentials $(\xi_+, \xi_-)$ with $\xi_+ = F_+^{-1} (n, m) F_+(n+1, m)$
 and $\xi_- =G_-^{-1} (n, m) G_-(n, m+1)$. It will be shown that 
 $\xi_+$ and $\xi_-$ depend only on $n \in \Z$ and $m \in \Z$, respectively, 
 see Theorem \ref{thm:potential}. Conversely using the loop group decomposition and 
 a pair of ordinary difference equations $F_+(n+1) = F_+(n) \xi_+(n)$ and 
 $G_-(m+1) = G_-(m) \xi_-(m)$, one has the discrete extended frame $F$ of 
 the original discrete {\ps} surface, see Theorem \ref{thm:converseconstruction}.
 Finally derivative of the discrete extended Frame reproduces the original discrete 
 {\ps} surface.
 We call this construction a {\it nonlinear d'Alembert formula} for 
 discrete {\ps} surfaces.
 Moreover using the nonlinear d'Alembert formula, we consider two 
 examples of discrete {\ps} surfaces in detail. 
 It is interesting that difference equations 
 obtained by these examples (see equations \eqref{eq:DrotsinG} and \eqref{eq:dPIII}) 
 are the same difference equations 
 obtained by previous works \cite{RGH, Hoffmann}. 
 
 This paper is organized as follows:
 In Section \ref{sc:Pre}, basic results about smooth/discrete 
 {\ps} surfaces will be recalled. In particular we will recall
 the nonlinear d'Alembert formula for smooth {\ps} surfaces in Theorem  
 \ref{thm:dAlembert}.
 In Section \ref{sc:PS}, we will give a nonlinear d'Alembert formula 
 for discrete {\ps} surfaces using loop group decompositions.
 Theorems \ref{thm:potential} and \ref{thm:converseconstruction} 
 are the main results of this paper.
 In Section \ref{sc:Exam},
 two examples of discrete {\ps} surfaces 
 are discussed in detail. 
 In Appendix \ref{sc:App}, we will discuss a relation between a smooth/discrete flow of 
 constant torsion space curves and a smooth/discrete {\ps} surface in $\E^3$.
\section{Preliminaries}\label{sc:Pre}
 In this section, we recall basic notation and results about 
 smooth and discrete {\ps} surfaces in $\E^3$ \cite{MS, BP}
 and the nonlinear d'Alembert formula for smooth {\ps} surfaces \cite{Toda, DIS, 
 Krichever}.
 Throughout this paper, we use the Pauli matrices $\sigma_j \;(j=1, 2, 3)$ as follows:
\eq{
 \sigma_1=
\begin{pmatrix}
0 & 1 \\
1 & 0
\end{pmatrix},\;\;
\sigma_2=
\begin{pmatrix}
0 & -i \\
i & 0
\end{pmatrix}
\;\;\mbox{and}\;\;
 \sigma_3 = 
\begin{pmatrix}
1 & 0 \\
0 & -1
\end{pmatrix}.
}
\subsection{Pseudospherical surfaces}\label{subsc:pssurf}
 We first identify $\E^3$ with the Lie algebra 
 of the special unitary group $\SU$, which will be denoted by $\su$: 
\eq{
  (x, y, z)^t \in \E^3 \longleftrightarrow
 \frac{i}{2}x \sigma_1 - \frac{i}{2}  y \sigma_2 + \frac{i}{2} 
 z \sigma_3 \in \su. 
}
 Let $f$ be a {\ps} surface with Gaussian curvature $K=-1$ 
 in $\E^3$. 
 It is known that there exist the {\it asymptotic Chebyshev coordinates} 
 $(x, y) \in \R^2$ for $f$, that is, they are
 asymptotic coordinates normalized by $|f_x| = |f_y| =1$. 
 Here the subscripts $x$ and $y$ denote the 
 $x$- and $y$-derivatives $\partial_x$ and $\partial_y$, 
 respectively. 
 Let $\{e_1, e_2, e_3\}$ be the Darboux frame rotating on the tangent plane clockwise angle $u$, 
 where $u$ is the angle between two asymptotic lines.
 Then there exists a $\widetilde F$ taking values in $\SU$ such that
\begin{equation}\label{eq:Darbouxframe}
 e_1 =  - \frac{i}{2} \widetilde F\sigma_1\widetilde F^{-1}, \;\;
 e_2 =  -\frac{i}{2} \widetilde F\sigma_2\widetilde F^{-1} \;\;
 \mbox{and}\;\;
 e_3 = - \frac{i}{2} \widetilde F\sigma_3\widetilde F^{-1},
\end{equation}
 see for example \cite{DIS}. Without loss of generality, at some base point 
 $(x_*, y_*) \in \R^2$, we have $\tilde F(x_*, y_*)=\id$.
 Then there exists a family of frames $F$ parametrized by $\l \in \R_+:=
\{ r \in \R \;|\; r>0\}$ 
 satisfying  the following system of partial differential equations:
\begin{equation}\label{eq:Laxpair}
 F_x = F U \;\;\mbox{and} \;\;F_y = F V,
\end{equation}
 where 
\begin{equation}\label{eq:movingUV}
 U= 
\frac{i}{2}
 \begin{pmatrix}
- u_x   &  \l \\
  \l    & u_x
 \end{pmatrix}, \;\;
 V= -\frac{i}{2} 
 \begin{pmatrix}
 0  &  \l^{-1} e^{i u} \\
 \l^{-1} e^{-i u}     & 0
 \end{pmatrix}.
\end{equation}
 The parameter $\lambda \in \R_{+}$ will be called the {\it spectral parameter}.
 We choose $F$ such that 
\eq{
F|_{\l=1} = \widetilde F \;\;\mbox{and}\;\;
F|_{(x_*, y_*)}  = \id.
}
 The compatibility condition of the system in \eqref{eq:Laxpair}, 
 that is $U_y -V_x +[V, U] =0$, 
 becomes a version of the sine-Gordon  equation:
\begin{equation}\label{eq:sine-Gordon}
 u_{x y} - \sin u =0.
\end{equation}
 It turns out that the sine-Gordon equation is the Gauss-Codazzi equations for 
 {\ps} surfaces. Thus from the fundamental theorem of surface theory
 there exists a family of {\ps} surfaces 
 parametrized by the spectral parameter $\l \in \R_+$, 
 see \cite{DIS} for a clear explanation of it.

 Then it is easy to see that 
 $F$ defined in \eqref{eq:Laxpair} together with the condition 
 $F|_{(x_*, y_*)}  = \id$ is an element in the {\it twisted $\SU$-loop group}:
\begin{equation}\label{eq:SUloop}
\LSU := \left\{ g : \R^{\times} \cup S^1 \to \SL\;|\; 
\mbox{$g$ is smooth, $g(\lambda) = \overline{g(\bar \lambda)}^{t -1}$ and $\sigma g(\lambda) = g(-\l)$}\right\},
\end{equation} 
 where $\sigma X = \ad (\sigma_3) X = \sigma_3 X \sigma_3^{-1}, 
 (X \in \SL)$  is an involution on $\SL$. 
 In order to make the above 
 group a Banach Lie group, we restrict the occurring matrix coefficients to 
 the Wiener algebra $\mathcal A =\{ f(\l)= \sum_{n \in \Z} f_n \l^n : S^1
\rightarrow \mathbb C \;|\;  \sum_{n \in \Z} |f_n| < \infty \}$.
 Then the Wiener algebra is a Banach algebra relative to the norm $\|f\| = \sum |f_n|$
 and the loop group $\LSU$ is a Banach Lie group, \cite{Gohberg}.

 Then the family of frames 
 $F$ will be called the {\it extended frame} for $f$. 
 We note that in fact the extended frame $F$ of a smooth {\ps} surface
 is an element in the subgroup of $\LSU$, consisting of loops which extend to 
 $\C^{\times}:= \C \setminus \{0\}$ 
 as analytic functions of $\lambda$.
 
 From the extended frame $F$, a family of {\ps} surfaces $f^{\lambda},\;
 (\l \in \R_+)$ 
 is given by the so-called {\it Sym formula}, \cite{Sym}:
 \begin{equation}\label{eq:Symformula}
 f^{\lambda} = \lambda \left.  \frac{\partial F}{\partial \lambda} 
 F^{-1}\right|_{\l  \in \R_{+}}.
 \end{equation}
\begin{Remark}
\mbox{}
\begin{enumerate}
\item The immersion $f^{\l}|_{\l=1}$ is the original {\ps} surface $f$
 up to rigid motion. The one-parameter family $\{f^{\l}\}_{\l \in \R_{+}}$ 
 will be called the {\it associated family} of $f$.
\item It is known that a surface of negative Gauss curvature 
 is a {\ps} surface if and only if the unit normal 
 of the surface (Gauss map) is Lorentz harmonic. Note that the 
 Lorentz structure is induced by the second fundamental form of a surface. 
\end{enumerate}
\end{Remark}

\subsection{Nonlinear d'Alembert formula}
 In this subsection, we recall the nonlinear d'Alembert formula for smooth 
 {\ps} surfaces in \cite{Krichever, Toda, DIS}. 
 Let $\D^{+}$ and $\D^{-}$ be the interior of the unit disk in 
 the complex plane and the union of the exterior of the unit disk in the complex plane 
 and infinity, respectively. 
 We first define two subgroups 
 of $\LSU$:
\begin{eqnarray}
 \LSUP  =\left\{ g \in\LSU \;|\;\mbox{$g$ can be analytically extend 
 to $\D^+$} \right\}, \\
 \LSUN  =\left\{ g \in\LSU \;|\;\mbox{$g$ can be analytically 
 extend to $\D^{-}$} \right\}.
\end{eqnarray}
 Then $\LSUPN$ and $\LSUNN$ denote subgroups 
 of $\LSUP$ and $\LSUN$ normalized at $\l =0$ and $\l = \infty$, respectively:
\eq{
 \LSUPN =\left\{ g \in \LSUP \;|\; g(\l =0) = \id \right\} 
 \;\;\mbox{and}\;\;
 \LSUNN =\left\{ g \in \LSUN \;|\; g(\l =\infty) = \id \right\}. 
}
 The following decomposition theorem is fundamental.
\begin{Theorem}[Birkhoff decomposition, \cite{Gohberg, Brander}]
 \label{thm:Birkhoff}
 The multiplication maps 
\eq{
 \LSUPN \times \LSUN \to \LSU \;\;\mbox{and}\;\; 
 \LSUNN \times \LSUP \to \LSU
}
 are diffeomorphisms onto $\LSU$, respectively.
\end{Theorem}
 We note that the Birkhoff decomposition is in general 
 only diffeomorphism onto the open dense subset. However, 
 in \cite{Brander}, it was proved that if the group is compact semisimple 
 then the diffeomorphism extends to the everything.

 From now on, for simplicity, we assume that the base point is $(x_*, y_*) =(0, 0)$ and 
 the extended frame $F$ at the base point is identity:
\eq{
 F(0, 0, \lambda) = \id.
}
 The nonlinear d'Alembert formula for {\ps} surfaces is
 summarized as follows.
\begin{Theorem}[\cite{Toda, DIS}]\label{thm:dAlembert}
 Let $F$ be the extended frame for a {\ps} surface $f$ in $\E^3$.
 Moreover, let $F= F_{+} F_{-}$ and $F= G_{-} G_{+}$ be the Birkhoff 
 decompositions given in Theorem \ref{thm:Birkhoff}, respectively.
 Then $F_{+}$ and $G_{-}$ do not depend on $y$ and $x$, respectively, and 
 the Maurer-Cartan forms of $F_{+}$ and $G_{-}$ are given as follows$:$
\begin{equation}\label{eq:potential} 
\left\{
\begin{array}{l}
 \xi_+ = \displaystyle F_{+}^{-1} d F_{+} 
 =\frac{i}{2} \l 
\begin{pmatrix}
 0 & e^{-i \alpha(x)} \\
 e^{i \alpha(x)} & 0
\end{pmatrix} dx, \\[0.3cm]
\xi_- = G_{-}^{-1} d G_{-} 
 =\displaystyle
 -\frac{i}{2}\l^{-1} 
\begin{pmatrix}
 0 & e^{i \beta(y)} \\
 e^{-i \beta(y)} & 0
 \end{pmatrix}dy,
\end{array}
\right.
\end{equation}
 where, using the angle function $u(x, y)$, $\alpha$ and $\beta$ are given by
\eq{
 \alpha(x) = u(x, 0)-u(0, 0)\;\;\mbox{and}\;\;\beta(y) 
 = u(0, y).
}
 Conversely, let $\xi_{\pm}$ be a pair of $1$-forms defined in \eqref{eq:potential}
 with functions $\alpha(x)$ and $\beta(y)$ satisfying $\alpha(0)=0$.
 Moreover, let $F_{+}$ and $G_{-}$ be solutions of 
 the pair of following ordinary differential equations$:$
 \begin{equation*}
\left\{
\begin{array}{l}
 d F_{+} = F_{+} \xi_{+}, \\
 d G_{-} = G_{-} \xi_{-},
\end{array}
\right.
\end{equation*}
 with $F_{+} (x =0, \l) = G_{-}(y =0, \l) = \id$.
 Moreover let $D = \di (e^{-\frac{i}{2} \alpha}, e^{\frac{i}{2} \alpha})$ and 
 decompose $(F_{+}D)^{-1} G_{-}$ by the Birkhoff decomposition 
 in Theorem \ref{thm:Birkhoff}$:$
\eq{
(F_{+}D)^{-1} G_{-} = V_{-} V_{+}^{-1},
}
 where $V_{-} \in \LSUNN$ and $V_{+} \in \LSUP$.
 Then $F = G_{-} V_{+} = F_{+} DV_{-}$ is the extended frame of some {\ps} surface in $\E^3$.
\end{Theorem}
 The pair of $1$-forms $(\xi_+, \xi_-)$ in \eqref{eq:potential}
 will be called the {\it pair of normalized potentials}.

 \begin{Remark}
 In \cite{DIS}, it was shown that the extended frames of {\ps} surfaces 
 can be also constructed from the following pair of $1$-forms:
 \begin{equation}\label{eq:genepot}
\eta^x = \sum_{j = - \infty}^{1} \eta^x_{j} \lambda^{j} dx
 \;\;\mbox{and}\;\;
\eta^y = \sum_{j = - 1}^{\infty} \eta^y_{ j} \lambda^{j} dy, 
\end{equation}
 where $\eta_j^x$ and  $\eta_j^y$ take values in $\su$, 
 and each entry of $\eta_j^x$ (resp. $\eta_j^y$) is smooth on $x$ (resp. $y$), 
 and $\det \eta^x_1 \neq 0,  \det \eta^x_{-1} \neq 0$. Moreover $\eta_j^x$ 
 and $\eta_j^y$ are diagonal (resp. off-diagonal) if $j$ is even (resp. odd).
 This pair of $1$-forms $(\eta^x, \eta^y)$ is a generalization 
 of the normalized potentials $(\xi_+, \xi_-)$ in \eqref{eq:potential} 
 and will be called the {\it pair of generalized potentials}, see also \cite{BIK}.
\end{Remark}
\subsection{Discrete pseudospherical surfaces}
 It is known that a discrete analogue of {\ps} surfaces and its 
 loop group formulation was defined in \cite{BP}. 
 In this subsection, instead of the smooth coordinates 
 $(x, y) \in \R^2$, 
 we use the quadrilateral lattice $(n, m) \in \Z^2$, that is,
 all functions depend on the lattice $(n, m) \in \Z^2$.
 The subscripts $1$ and $2$ (resp. $\bar 1$ and $\bar 2$)
 denote the forward (resp. backward) lattice points with 
 respect to $n$ and $m$: For a function $f(n, m)$ of the lattice 
 $(n, m) \in \Z^2$, we define $f_1, f_2, f_{\bar 1}$ and $f_{\bar 2}$  by
\eq{
 f_{1} = f(n+ 1, m), \; f_{\bar 1} = f(n- 1, m),\;
 f_{2} = f(n, m+ 1) \;\;\mbox{and}\;\; f_{\bar 2} = f(n, m-1).
}
 The {\it discrete extended frame} $F$ of a discrete {\ps} surface 
 can be defined by  the following partial difference system, 
 see \cite[Section 3.2]{BP:surv} and 
 \cite{BP} 
 \footnote{Since the identification of $\E^3$ and $\su$ is different in \cite{BP:surv}, 
 the matrices $\mathcal U$ and $\mathcal V$ of $(3.22)$ and $(3.23)$ in \cite{BP:surv}
 and our $U, V$ are related by 
 $U = \ad(\offdi(1, 1))\mathcal {}^t\mathcal U$ and $V = \ad(\offdi(1, -1)){}^t \mathcal V$ 
 with $h_{n, m} = u/2, 
 p_n = p/2, q_m = q/2$. On the one hand if we normalize $p=q= \delta$, then $U$ and $V$
 are the same as $\Omega_{n,m}$ and $\Theta_{n, m}$ in \cite[Theorem 3.1]{PW}. }:
\begin{equation}\label{eq:DLaxpair}
 F_1 = F U \;\;\mbox{and} \;\;F_2 = F V,
\end{equation}
 where 
\begin{equation}\label{eq:DmovingUV}
 U = 
\frac{1}{\Delta_{+}}
 \begin{pmatrix}
 e^{-\frac{i}{2} (u_1 -u)}  &  \frac{i}{2}p \lambda  \\
 \frac{i}{2} p \lambda     &  e^{\frac{i}{2}(u_1 -u)}
 \end{pmatrix},  \;\;
 V= \frac{1}{\Delta_{-}} 
 \begin{pmatrix}
 1  &  -\frac{i}{2} q e^{\frac{i}{2}(u_2 + u)} \l^{-1}\\
- \frac{i}{2} q e^{-\frac{i}{2} (u_2 + u)} \l^{-1}& 1
 \end{pmatrix},
\end{equation}
 with $\Delta_+ = \sqrt{1 + (p/2)^2 \l^2}$ and 
 $\Delta_- = \sqrt{1 + (q/2)^2 \l^{-2}}$. Here 
 $u$ is a real function depending on both $n$ and $m$, and 
 $p\neq 0$ and $q\neq 0$ are real functions depending only on $n$ and $m$, 
 respectively: 
\eq{
 u = u(n, m), \;\;
 p = p(n), \;\;\mbox{and}\;\;
 q = q(m).
 }
 The compatibility condition of the system in \eqref{eq:DLaxpair},
 that is $V U_2 = UV_1$, 
 gives the so-called {\it discrete sine-Gordon equation}:
\begin{equation}\label{eq:dsineGordon}
 \sin \left( \frac{u_{12}-u_1-u_2+u}{4}\right)
 = \frac{p q}{4} \sin \left( \frac{u_{12}+u_1 + u_2+u}{4}\right).
\end{equation}
 The equation \eqref{eq:dsineGordon} was first found by Hirota in \cite{Hirota}
 and also called the {\it Hirota equation}.
 Then a discrete {\ps} surface $f$ can be given by the so-called Sym formula, 
 \cite{BP}:
 \begin{equation}\label{eq:dSymformula}
 f^{\lambda} = \lambda \left.  \frac{\partial F}{\partial \lambda} 
 F^{-1}\right|_{\l  \in \R_{+}}.
 \end{equation}
 The family of frames $F$
 defined by \eqref{eq:DLaxpair} with $F|_{(n_*, m_*)} = \id$ will be called 
 the {\it discrete extended frame} for a discrete {\ps} surface.
\begin{Remark}\label{Rm:DPS}
 The geometric characterization of a discrete {\ps} surface $f$
 was also given in \cite{BP, BP:surv}.
 A discrete surface $f$ is a discrete {\ps} surface if and only if
 the following two conditions hold:
 \begin{enumerate}
\item For each point $f$, there is a plane $P$ such that 
\eq{
 f, \; f_1, \; f_{\bar 1}, \; f_2, \; f_{\bar 2} \in P.
} 
\item The length of the opposite edge of an elementary quadrilateral is equal:
\eq{
 |f_1 - f| = |f_{1 2} -f_2| = a(n) \neq 0, \;\;
 |f_2 - f| = |f_{1 2} -f_1| = b(m) \neq 0.
} 
\end{enumerate}
 It is easy to see that the map $f^{\lambda}$ 
 defined in \eqref{eq:dSymformula} has these properties and $f^{\l}$ gives 
 a family of discrete {\ps} surfaces, see \cite[Theorem 3]{BP:surv}.
\end{Remark}
 We now introduce a discrete analogue of the smooth angle function between all edges 
 at point $(n, m) \in \Z^2$
 as follows, see \cite{BP}:
\eq{
 \phi^{(1)} &\equiv  - \frac{u_2}{2}-\frac{u_1}{2}, \;\;\;\; \phi^{(2)} \equiv \frac{u_{\bar 1}}{2} + \frac{u_{2}}{2},\\
  \phi^{(3)} & \equiv  - \frac{u_{\bar 2}}{2} - \frac{u_{\bar 1}}{2},\;\;\;\;
  \phi^{(4)} \equiv \frac{u_{1}}{2} + \frac{u_{\bar 2}}{2},
}
 where $\equiv$ means that the equalities hold modulo $2 \pi$.
 Then from the symmetry of the quadrilateral and the fact that the sum of angles 
 around a vertex is equal to $2 \pi$ by (1) in Remark \ref{Rm:DPS}, 
 we have
\eq{
 \phi^{(1)} = \phi^{(3)}_{12},\; \phi^{(2)}_{1} = \phi^{(4)}_{2} \;\;\mbox{and}\;\;
 \phi^{(1)} +\phi^{(2)}+ \phi^{(3)} +\phi^{(4)} \equiv 0.
 }
 Moreover, 
 setting $Q =\exp (i \phi^{(1)})$ and $k = p q/4$, the discrete sine-Gordon equation 
 \eqref{eq:dsineGordon} can be rephrased as 
\begin{equation}\label{eq:danothersine-Gordon}
 Q_{12} Q =\frac{Q_2-k_2}{1-k_2 Q_2} \frac{Q_{1} - k_1}{1-k_1 Q_1},
\end{equation}
 see for example \cite[Section 3.4]{BP:surv}.
 We note that the other discrete angle functions $\exp (i \phi^{(j)}) \;(j \in \{2, 3, 4\})$ 
 also satisfy the same equation \eqref{eq:danothersine-Gordon}.
\section{Nonlinear d'Alembert formula 
 for discrete {\ps} surfaces}\label{sc:PS}
 In this section, we derive a discrete version of nonlinear d'Alembert formula 
 for discrete {\ps} surfaces. 
%Moreover, we consider the 
% discrete {\ps} surfaces as a flow of a discrete constant 
% torsion curve analogous to the smooth case. Many properties 
% of a discrete {\ps} surface can be reinterpreted as the 
% constraint of the flow of the constant torsion curves.
% \footnote{I still need to change.}

\subsection{Nonlinear d'Alembert formula}\label{subsc:DDAlembert}
 From now on we assume that the base point is $(n_*, m_*)=(0, 0)$
 and the discrete extended frame $F$ at the base point is identity:
\eq{
 F(0, 0, \lambda) = \id.
}
 Moreover, we also assume that the functions $p$ and $q$ in \eqref{eq:DmovingUV}
 satisfy the inequalities
\begin{equation}\label{eq:conditions}
 0 < \left|\frac{p}{2}\right|<1 \;\;\mbox{and}\;\;0<\left|\frac{q}{2}\right|<1.
\end{equation}
 These inequalities imply that the zeros of $\Delta_+$ and $\Delta_{-}$ are 
 in the exterior and the interior of the unit disk, respectively.
 Moreover it is easy to see that zeros of $\Delta_{\pm}$ are not on $\R\cup \{\infty\}$.
 Then the discrete extended frame $F$ takes values in the twisted loop group $\LSU$.

 The following theorem gives a pair of matrices, which 
 depend only on one lattice variables and it is a discrete version 
 of Theorem \ref{thm:dAlembert}.
 \begin{Theorem}\label{thm:potential}
 Let $f$ be a discrete {\ps} surface and $F$ the corresponding 
 discrete extended frame. Decompose $F$ according 
 to the Birkhoff decomposition in Theorem \ref{thm:Birkhoff}$:$
\eq{
 F= F_{+} F_{-} = G_{-} G_{+},
}
 where $F_{+} \in \LSUPN, F_{-} \in \LSUN, G_- \in \LSUNN$ 
 and $G_+ \in \LSUP$.
 Then $F_{+}$ and $G_-$ do not depend on $m \in \Z$ and 
 $n \in \Z$, 
 respectively, and the discrete Maurer-Cartan forms 
 of $F_+$ and $G_-$ are given 
 as follows:
\begin{equation}\label{eq:discretepotentials}
\left\{
\begin{array}{l}
\displaystyle
 \xi_{+} = F_{+}^{-1} (F_{+})_1 = 
 \frac{1}{ \Delta_{+}}
\begin{pmatrix}
1 & \frac{i}{2} p e^{-i \alpha}\l \\
\frac{i}{2} p e^{i \alpha} \l &1 
\end{pmatrix}, \\[0.3cm]
\displaystyle
 \xi_{-} = G_{-}^{-1} (G_{-})_2 = 
 \frac{1}{ \Delta_-}
\begin{pmatrix}
1 &  -\frac{i}{2} q e^{i \beta}\l^{-1} \\
-\frac{i}{2} q e^{-i \beta}\l^{-1} &1
\end{pmatrix},
\end{array}
\right.
\end{equation}
 where $\Delta_{+} = \sqrt{1 + (p/2)^2 \l^2}$ 
 and $\Delta_{-} = \sqrt{1 + (q/2)^2 \l^{-2}}$, 
 the functions $p$ and $q$ are given in \eqref{eq:DmovingUV},
 and $\alpha$ and $\beta$ are functions of $n \in \Z$ and $m \in \Z$, 
 respectively.
 Moreover using the function $u(n, m)$ in \eqref{eq:DmovingUV}, 
 $\alpha(n) $ and $\beta(m)$ are given by
 \begin{equation}\label{eq:potfunctions}
\left\{
\begin{array}{l}
 \alpha(n) =\frac{1}{2} u(n+1, 0)+\frac{1}{2}u(n, 0)- u(0, 0) , \\[0.1cm]
 \beta(m) = \frac{1}{2}u(0,m+1)+\frac{1}{2}u(0, m).
\end{array}
\right.
 \end{equation}
\end{Theorem}
\begin{proof}
 Let  $F = F_+ F_- = G_- G_+$ be the Birkhoff decomposition
 of the discrete extended frame. 
 Since $F_+$ takes values  in $\LSUPN$,
 $F_{+}^{-1} (F_{+})_2$  also takes values in $\LSUPN$.
 On the one hand, we have  
\eq{
 F_{+}^{-1} (F_{+})_2 = F_- F^{-1} F_2 (F_{-})_2^{-1}
 = F_- V (F_{-})_2^{-1}.
}
 Since $V$ takes values in $\LSUNN$,  $F_- V (F_{-})_2^{-1}$ 
 can be expanded with respect 
 to $\lambda$ as
\eq{
 F_- V (F_{-})_2^{-1} = 
 V_{(0)} + \lambda^{-1} V_{(-1)} + \lambda^{-2}V_{(-2)} 
 + \lambda^{-3} V_{(-3)} + \cdots.
}
 Thus 
 \begin{equation}\label{eq:V0}
 F_{+}^{-1} (F_{+})_2 = \id. 
\end{equation}
 Therefore $F_{+}$ depends only on $n$.
 Similarly $F_{+}^{-1} (F_{+})_1$ takes values in $\LSUPN$ and 
 $F_{+}^{-1} (F_{+})_1 = F_- F^{-1} F_1 (F_{-})_1^{-1}
 = F_- U (F_{-})_1^{-1}$. 
  Multiplying $\Delta_+$ on both sides, we have
 \begin{equation}\label{eq:U0}
 \Delta_+ F_{+}^{-1} (F_{+})_1 = U_{(0)} + \lambda U_{(1)} 
 = F_- (\Delta_+  U) (F_{-})_1^{-1},
\end{equation}
 where we use the expansion $F_- (\Delta_+U) (F_{-})_1^{-1} 
 = \l U_{(1)} + U_{(0)} + \l^{-1} U_{(-1)} + \cdots$. 
 Since $F_+$ takes values in $\LSUPN$, 
 $U_{(0)}=\id$ and $U_{(1)}$ has the off diagonal 
 form with two entries are 
 minus complex conjugate each other and 
 $\det (U_{(0)} + \l U_{(1)}) =\Delta_+^2$. Then
\begin{equation}\label{eq:xiplus} 
  \xi_+ = F_{+}^{-1} (F_{+})_1 = 
 \frac{1}{\Delta_+}
\begin{pmatrix} 
1 &  \ell(n)\l \\
-\overline{\ell (n)} \l &1
\end{pmatrix}, \;\;\mbox{with}\;\;
 |\ell(n)|^2 =p(n)^2.
\end{equation}
 Thus we have the form $\xi_+$ in \eqref{eq:discretepotentials}. 
 We now look the expansion of $F_- (\Delta_+ U) (F_{-})_1^{-1}$ 
 more closely. Since $F_{-} \in \LSUN$, we can set $F_-(\lambda =\infty) 
 = \di (e^{\frac{i}{2} f}, e^{-\frac{i}{2} f})$. Then from \eqref{eq:V0}, we 
 see that 
\eq{
 \frac{f_2}{2} = \frac{f}{2} \;\;(\operatorname{mod} 2\pi).
}
 Then consider \eqref{eq:U0} multiplying $\Delta_+^{-1}$ on 
 both sides and set $m=0$:
\begin{equation}\label{eq:xi+}
 \xi_+ = F_+^{-1} (F_+)_1 = F_+^{-1} (F_+)_1|_{m=0} = F_- U (F_-)_1^{-1}|_{m=0}.
\end{equation}
 Let $\tilde F_+$ be the solution of $\tilde F_+^{-1} 
 (\tilde F_+)_1 =U|_{m=0}$ with $\tilde F_+(n=0) = \id$. It is easy to see 
 that $\tilde F_+$ does not depend on $m \in \Z$ and
 takes values in $\LSUP$. 
 Moreover, $\tilde F_+ F_-^{-1}|_{m=0}$ satisfies that 
 $( \tilde F_+F_-^{-1})^{-1}  (\tilde F_+F_-^{-1})_1 |_{m=0} =
 F_- U (F_-)_1^{-1}|_{m=0}$.
 Since $\tilde F_+ F_-^{-1}|_{(n,m) =(0, 0)} = F_+|_{(n,m) =(0, 0)} = \id$ and 
 by \eqref{eq:xi+}, we have  
\eq{
 \tilde F_+ F_-^{-1}|_{m=0} =  F_+|_{m=0}. 
}
 Thus $F_-|_{m=0} (= F_+^{-1} \tilde F_+ |_{m=0})$  takes values in 
 $\LSUP \cap \LSUN$ and therefore 
 $F_-|_{m=0}$ does not depend on $\lambda$. A straightforward computation 
 shows that 
\eq{
 F_- U (F_-)_1^{-1}|_{m=0} =
\left. 
\frac{1}{\Delta_+}
 \begin{pmatrix} 
e^{-\frac{i}{2}(u_1 -u) -\frac{i}{2}(f_1 -f)}&\frac{i}{2} p e^{\frac{i}{2} (f_1+f)}\l \\ 
 \frac{i}{2}p e^{- \frac{i}{2}(f_1+f)}\l & e^{\frac{i}{2}(u_1 -u) +\frac{i}{2} (f_1 -f)} 
 \end{pmatrix}
 \right|_{m=0}.
}
 Thus comparing this with $\xi_+$ in \eqref{eq:xiplus} 
 and setting $\alpha =-\frac{1}{2}(f_1+ f)|_{m=0}$, 
 we have $-(u_1 -u)|_{m=0} = (f_1 -f)|_{m=0}$ and $\ell = \frac{i}{2}p e^{-i \alpha}$.
 The equation $-(u_1 -u)|_{m=0} = (f_1 -f)|_{m=0}$ can be easily solved to $f$ with 
 the initial condition $f(0, 0) =0$:
\eq{
 f(n,0) = u(0,0)-u(n,0).
}
 Note that $f(0,0) =0$ is satisfied since $F_-(\l = \infty) = \id$.
 Therefore $\alpha$ can be  explicitly computed as 
\eq{
 \alpha(n) = \frac{1}{2}u(n+1, 0) +\frac{1}{2}u(n, 0)-u(0,0).
}
 For the case of $G_{-}$, the verbatim argument shows that 
 $G_-$ does not depend on $n$ and $\xi_- = G_{-}^{-1}(G_{-})_2$ has the 
 following form:
\eq{
 \xi_- = G_{-}^{-1} (G_{-})_2 = \frac{1}{\Delta_-} 
\begin{pmatrix}
 1 & r (m) \l^{-1} \\
 - \overline{r (m)} \l^{-1} &1\\
\end{pmatrix}\;\;\mbox{with}\;\; |r(m)|^2 = q(m)^2.
}
 Thus we have the form $\xi_-$ in \eqref{eq:discretepotentials}. 
 Consider 
 \begin{equation}\label{eq:G-}
 \xi_{-} = G_{-}^{-1} (G_{-})_2|_{n=0} = G_{+} V (G_{+})^{-1}_2|_{n=0}.
 \end{equation}
 Let $\tilde G_-$ be the solution of $\tilde G_-^{-1} 
 (\tilde G_-)_2 =V|_{n=0}$ with $\tilde G_-(m=0) = \id$.
 It is easy to see that $\tilde G_-$ does not depend on $n \in \Z$ 
 and takes values in $\LSUN$. 
 Moreover, $\tilde G_- G_+^{-1}|_{n=0}$ satisfies 
 $G_+ \tilde G_-^{-1}  (\tilde G_-)_2 (G_+)_2^{-1}|_{n=0} =
 G_+ V (G_+)_2^{-1}|_{n=0}$.
 Since $\tilde G_+G_+^{-1}|_{(n,m) =(0, 0)}= G_{-}|_{(n, m)=(0,0)} = \id$ 
 and by \eqref{eq:G-}, we have 
\eq{
 \tilde G_- G_+^{-1}|_{n=0} =  G_-|_{n=0}. 
}
  Therefore $G_+|_{n=0} (=  G_-^{-1} \tilde G_-|_{n=0})$ takes values in 
 $\LSUP \cap \LSUN$ and thus $G_+|_{n=0}$ does not depend on $\lambda$. 
 Moreover, since  $G_{-}^{-1} (G_{-})_2$ and $V$ are in $\LSUNN$ and 
 from \eqref{eq:G-}, 
 we have $G_+|_{n=0} = \id$.
 A straightforward computation shows that 
\eq{
 G_+ V (G_+)_2^{-1}|_{n=0} =
\left. 
\frac{1}{\Delta_-}
 \begin{pmatrix} 
 1 &-\frac{i}{2}q e^{\frac{i}{2}(u_2 +u)}\l^{-1} \\ 
 -\frac{i}{2}q e^{-\frac{i}{2}(u_2+u)}\l^{-1} & 1
 \end{pmatrix}
 \right|_{n=0}.
}
 Thus setting $\beta =\frac{1}{2} (u_2 + u)|_{n=0}$, we have
 $r =-\frac{i}{2}q e^{i \beta}$. Therefore we have 
\eq{
 \beta(m) = \frac{1}{2}u(0,m+1)+\frac{1}{2}u(0, m).
}
 This completes the proof.
\end{proof}
\begin{Definition}
 The pair of matrices $(\xi_{-}, \xi_{+})$ given in \eqref{eq:discretepotentials}
 will be called the
 {\it pair of  discrete normalized potentials}.
\end{Definition}
 Conversely, the pair of discrete normalized potentials give some
 discrete {\ps} surface through the Birkhoff decomposition.
\begin{Theorem}\label{thm:converseconstruction}
 Let $\xi_{\pm}$ be a pair of discrete normalized potentials 
 defined in  \eqref{eq:discretepotentials} with arbitrary 
 functions $\alpha = \alpha(n),  \beta = \beta(m)$ with $\alpha (0) = 0$ 
 and $p = p(n), q = q(m)$ satisfying the conditions \eqref{eq:conditions}.
 Moreover let $F_{+} =F_+(n, \l)$ and $G_{-} =G_-(m, \l) $ be the solutions 
 of the ordinary difference equations
\eq{
 (F_{+})_{1} = F_{+} \xi_{+}\;\;\mbox{and}\;\;
 (G_{-})_{2} = G_{-} \xi_{-},
}
 with $F_{+}(n =0, \l) =G_{-}(m=0, \l)  = \id$ and set a matrix
 $D= \di (e^{\frac{i}{2} k}, e^{-\frac{i}{2} k}) \in {\rm U}_1$, 
 where $k(0) =0$ and $k(n) = 2 \sum_{j=0}^{n-1}(-1)^{j+n}\alpha(j)$ for $n\geqq 1$.
 Decompose $(F_{+}D)^{-1} G_{-}$ by the Birkhoff decomposition in 
 Theorem \ref{thm:Birkhoff}$:$
\begin{equation}\label{eq:Birkhoffforpot}
 (F_{+}D)^{-1} G_{-}  = V_{-} V_{+}^{-1},
\end{equation}
 where $V_{-} \in \LSUNN, V_{+} \in \LSUP$.
 Then $F = G_{-} V_{+} = F_{+} D V_{-}$ 
 is the discrete extended frame of some discrete {\ps} surface 
 in $\E^3$. Moreover the solution $u = u(n, m)$ of the discrete sine-Gordon 
 for the discrete {\ps} surface satisfies the relations in \eqref{eq:potfunctions}.
\end{Theorem}
\begin{proof} 
 Set $F= G_- V_+ = F_+ D V_-$ as in the statement of Theorem 
 \ref{thm:converseconstruction}. 
 Let us compute the discrete Maurer-Cartan form $F^{-1} F_1$ and 
 $F^{-1} F_2$, respectively. A direct computation shows that 
\eq{
 F^{-1} F_1 = V_-^{-1} D^{-1}\xi_+D_1  (V_-)_1 
 = V_+^{-1} G_-^{-1} (G_-)_1  (V_+)_1 = V_+^{-1} (V_+)_1.
}
 Here we use the property that $G_-$ does not depend on $n \in \Z$.
 Thus $F^{-1}F_1$ takes values in $\LSUP$ and using the expansion 
 $V_-^{-1} D^{-1}(\Delta_+\xi_+)D_1  (V_-)_1  = * \l+ *\l^0 + *\l^{-1} +\cdots$,
 we have 
\eq{
 \Delta_+ F^{-1} F_1 =  F_{(0)}  + \l F_{(1)}.
}
 Here $F_{(0)}$ and $F_{(1)}$ are 
 diagonal and off-diagonal matrices, respectively.
 Moreover, from the form of $\xi_+$ and $V_{+}(\lambda = 0) 
 = \di (e^{\frac{i}{2} h}, e^{-\frac{i}{2} h})$ and $\alpha + \frac{1}{2}(k_1 + k) =0$, 
 we have 
\begin{equation}\label{eq:MC1-1}
 F^{-1} F_1 = \frac{1}{\Delta_+} 
 \begin{pmatrix} 
 e^{\frac{i}{2} (h_1-h)} & \frac{i}{2} p  \l \\
 \frac{i}{2} p \l &   e^{-\frac{i}{2} (h_1-h)}
\end{pmatrix}.
\end{equation}
 Similarly we compute $F^{-1} F_2$ as
\eq{
 F^{-1} F_2 = V_+^{-1} \xi_-  (V_+)_2 = V_-^{-1} (F_+D)^{-1}(F_+D)_2(V_-)_2 
 =V_-^{-1}(V_-)_2.
}
 Here we use the property that $F_+D$ does not depend on $m \in \Z$.
 Thus, noting that $F^{-1} F_2 \in \LSUPN$, we have 
\eq{
 \Delta_- F^{-1} F_2 = \id + \l^{-1} F_{(-1)}.
}
 Since $F \in \LSU$,  $F_{(-1)}$ is 
 an off-diagonal matrix.
 Moreover, since $F$ is unitary on $\lambda \in \R_{+}$
 and the form of $\xi_-$, we have 
\begin{equation}\label{eq:MC1-2}
 F^{-1} F_2 = \frac{1}{\Delta_-}
 \begin{pmatrix} 
 1 & -\frac{i}{2} qe^{i \beta - \frac{i}{2}(h+h_2) }  \l^{-1} \\
- \frac{i}{2}q e^{ -i\beta + \frac{i}{2}(h+h_2) }\l^{-1} &  1
 \end{pmatrix}.
\end{equation}
 From the compatibility condition of $F$, that is $F_{12} = F_{21}$, we have
\begin{equation}\label{eq:dcompact}
 -X_2 + Y = X+ Y_1 \;\;\mbox{and}\;\;
e^{\frac{i}{2} X_2} + \frac{pq}{4} e^{\frac{i}{2} Y} =
e^{\frac{i}{2} X} + \frac{pq}{4} e^{-\frac{i}{2} Y_1},
\end{equation}
 where $X = h_1 - h$ and $Y = 2 \beta - (h_2 + h)$.
 The first equation in \eqref{eq:dcompact} can be solved by a function $u= u(n, m)$ as
\begin{equation}\label{eq:relationu}
 h_1 -h = - u_1 + u \;\;\mbox{and} \;\; 2 \beta - (h_2 + h)  = u_2 + u.
\end{equation}
 Then the second equation in \eqref{eq:dcompact} is the discrete sine-Gordon equation and 
 $F$ is the extended frame for some discrete {\ps} surface. 
 Let us look at the Birkhoff decomposition in \eqref{eq:Birkhoffforpot}
 with $n =0$:
\eq{
(F_+D)^{-1} G_-|_{n =0} = V_- V_+^{-1}|_{n=0}.
}
 Then since $F_+ D|_{n =0} = \id$ and $G_-$ takes values in $\LSUNN$, we have 
\eq{
V_-|_{n =0} = G_-|_{n =0}  \;\;\mbox{and} \;\; V_+|_{n=0}= \id.
}
 Thus $h|_{n=0} = 0$ and the second equation in \eqref{eq:relationu} 
 is equivalent with the second equation in \eqref{eq:potfunctions}. Similarly 
 consider the Birkhoff decomposition in \eqref{eq:Birkhoffforpot}
 with $m =0$:
\eq{
(F_+D)^{-1} G_-|_{m =0} = V_- V_+^{-1}|_{m= 0}.
}
 Then since $G_-|_{m =0} = \id$ and $F_+D$ takes values in $\LSUP$, we have 
\eq{
 V_-|_{m=0}= \id  \;\;\mbox{and} \;\; V_+|_{m=0} = F_+ D|_{m=0}.
}
 Thus $h|_{m=0} = k$ and from the first equation in \eqref{eq:relationu} 
 is equivalent with $k_1 -k = (- u_1 + u)|_{m =0}$. Since 
 $k(n) = 2 \sum_{j=0}^{n-1}(-1)^{j+n}\alpha(j)$, we can solve the equation 
 for $\alpha$, that is, $\alpha(n) = \frac{1}{2} u(n+1, 0) + \frac{1}{2} u(n,0) 
 -u(0, 0)$, which is the first equation in \eqref{eq:potfunctions}.
\end{proof}

\subsection{Pairs of discrete generalized potentials}
 In Section \ref{subsc:DDAlembert}, we discussed the 
 construction of discrete {\ps} surfaces from pairs of discrete normalized potentials.
 In this subsection, we generalize the pairs of discrete normalized potentials 
 to pairs of discrete generalized potentials and show that 
 those pairs produce also discrete {\ps} surfaces. 
 Let $\xi_{\pm}$ be a pair of discrete normalized potentials and
 set
\begin{equation}\label{eq:generalizedpot}
\etn = 
P_-^{l}\xi_+ P_-^{r}, 
\;\;
\etm =
P_+^l \xi_- P_+^r.
\end{equation}
 Here we assume that $P_{\pm}^{\star}$ ($\star = l$ or $r$)
 take values in $\LSUPM$  
 and do not depend on $m$ and $n$, respectively, that is, 
 $P_-^\star = P_-^\star(n, \l)$ and $P_+^\star = P_+^\star(m, \l)$.
 Thus the $\etn$ and $\etm$ do not depend on $m$ and $n$, respectively:
\eq{
 \etn = \etn (n, \l), \;\;
 \etm = \etm (m, \l).
}
 Then the pair of discrete generalized potentials 
 $(\etn, \etm)$ produces a discrete {\ps} surface as follows.
\begin{Theorem}\label{thm:converseconstruction2}
 Let $(\etn, \etm)$ be the pair of discrete 
 generalized potentials 
 defined in  \eqref{eq:generalizedpot}.
 Moreover let $\Fn$ and $\Gm$ be the solutions 
 of the ordinary difference equations
\eq{
 (\Fn)_{1} = \Fn \etn\;\;\mbox{and}\;\;
 (\Gm)_{2} = \Gm \etm,
}
 with some initial condition $\Fn(n =0)$ and $\Gm(m=0)$ taking values in $\LSU$ 
 and let $D = \di (e^{\frac{i}{2} k}, e^{-\frac{i}{2} k}) \in {\rm U}_1$, 
 where $k = k(n)$ is 
\eq{
 k(0) =0, \;\; k(n)=\sum_{j=0}^{n-1}(-1)^{j+n}(2 \alpha(j)- \theta^l(j) + \theta^r (j))
\;\;\mbox{for $n \geqq 1$}.
}
 Here $\theta^l$ and  $\theta^r$ 
 are defined by $P_-^{l}(n, \l = \infty) = \di (e^{\frac{i}{2} \theta^l(n)}, 
 e^{-\frac{i}{2} \theta^l(n)})$
 and $P_-^{r}(n, \l = \infty) = \di (e^{\frac{i}{2} \theta^r(n)}, 
 e^{-\frac{i}{2} \theta^r(n)})$, 
 respectively.
 Then decompose $(\Fn D)^{-1} \Gm$ by the Birkhoff decomposition in 
 Theorem \ref{thm:Birkhoff}$:$
\eq{
(\Fn D)^{-1} \Gm  = V_{-} V_{+}^{-1},
}
 where $V_{-} \in \LSUNN, V_{+} \in \LSUP$.
 Then $F = \Gm V_{+} = \Fn D V_{-}$ is the discrete 
 extended frame of some discrete {\ps} surface 
 in $\E^3$.
\end{Theorem}
\begin{proof} 
 Set $F= \Gm V_+ = \Fn D V_-$ as in the statement of Theorem 
 \ref{thm:converseconstruction2}. 
 Let us compute the discrete Maurer-Cartan form $F^{-1} F_1$ and 
 $F^{-1} F_2$, respectively. A direct computation shows that 
\eq{
 F^{-1} F_1 = V_-^{-1} D^{-1}\etn D_1  (V_-)_1 
 = V_+^{-1} \Gm^{-1} (\Gm)_1  (V_+)_1.
}
 Since $\Gm$ does not depend on $n$, the right hand side 
 of the equation becomes $V_+^{-1} (V_+)_1$ which takes values 
 in $\LSUP$.
 Therefore, $\Delta_+ F^{-1} F_1$ has the form 
\eq{
 \Delta_+ F^{-1} F_1 =  F_{(0)}  + \l F_{(1)}.
}
 From the twisted property of $\Fn$ and $\Gm$, $V_{\pm}$ and $F$
 also have the same property. Thus $F_{(0)}$ and $F_{(1)}$ are 
 diagonal and off-diagonal matrices, respectively.

 Moreover, from the form of $D^{-1}\etn D_1$ and using 
 $\alpha -\frac{1}{2}(\theta^l -\theta^{r})  + \frac{1}{2}(k_1 +k) = 0$, we have 
\eq{
 F^{-1} F_1 = \frac{1}{\Delta_+} 
 \begin{pmatrix} 
 e^{\frac{i}{2} (h_1-h)} & \frac{i}{2} p  \l \\
 \frac{i}{2} p \l &   e^{-\frac{i}{2} (h_1-h)}
\end{pmatrix},
}
 where we set $V_{+}(\lambda = 0) 
 = \di (e^{\frac{i}{2} h}, e^{-\frac{i}{2} h})$.
 Similarly we compute $F^{-1} F_2$ as
\eq{
 F^{-1} F_2 = V_+^{-1} \etm  (V_+)_2 = V_-^{-1} (F_nD)^{-1}(F_nD)_2 (V_-)_2 = 
 V_{-}^{-1} (V_{-})_2.
}
 Thus, noting that $V_-^{-1} (V_-)_2 \in \LSUPN$, we have 
\eq{
 \Delta_- F^{-1} F_2 = \id +  \l^{-1}F_{(-1)}.
}
 From the twisted property of $\Fn$ and $\Gm$, $V_{\pm}$ and $F$
 also have the same property. Thus $F_{(-1)}$ is 
 an off-diagonal matrix.
 Moreover, from the form of $\etm$, we have 
\eq{
 F^{-1} F_2 = \frac{1}{\Delta_-}
 \begin{pmatrix} 
 1 & -\frac{i}{2} qe^{i\beta - \frac{i}{2}(h+h_2 - \omega^l+ \omega^r) }  \l^{-1} \\
 -\frac{i}{2}q e^{-i \beta + \frac{i}{2}(h+h_2 - \omega^l+ \omega^r) }\l^{-1} &  1
 \end{pmatrix},
}
 where we set $P_+^l(\l =0) = \di (e^{\frac{i}{2} \omega^l}, e^{-\frac{i}{2} \omega^l})$
 and $P_+^r(\l =0) = \di (e^{\frac{i}{2} \omega^r}, e^{-\frac{i}{2} \omega^r})$.
 Since $F^{-1} F_1$ and $F^{-1} F_2$ have the forms in \eqref{eq:MC1-1} and \eqref{eq:MC1-2}, 
 respectively. Thus we can use the same argument in the proof of Theorem \ref{thm:converseconstruction}
 and there exists a function $u$ such that 
\eq{
 h_1 - h = -u + u, \;\;
 2 \beta - (h+ h_2 +\omega^{l} + \omega^r)= u_2 + u,
}
 and $u$ satisfies the discrete sine-Gordon equation.
 Then $F$ is the discrete extended frame of some discrete {\ps} surface in $\E^3$,
 which completes the proof.
\end{proof}
\begin{Definition}
 The pair $(\etn, \etm)$ given in \eqref{eq:generalizedpot} will be called 
 the {\it pair of discrete generalized potentials}.
\end{Definition}
\begin{Remark}
 The pair of normalized potentials $(\xi_+, \xi_-)$ and 
 the corresponding pair of discrete generalized potentials $(\etn, \etm)$ 
 in \eqref{eq:generalizedpot} give in general different discrete {\ps} 
 surfaces.
\end{Remark}
\section{Examples}\label{sc:Exam}
 In this section, we discuss two examples using the nonlinear d'Alembert 
 formula.
\subsection{Discrete pseudospherical surfaces of revolution}\label{sbsc:revo}
 We recall that the construction of smooth {\ps} surfaces of revolution 
 via the nonlinear d'Alembert formula. Let $(\eta^x, \eta^y)$ 
 be a pair of generalized potentials as follows:
\begin{equation}
 \eta^{x} = A dx, \;\; \eta^{y} = -A dy, \;\;\mbox{with}
 \;\;
 A = 
\begin{pmatrix} 
0 &i(\lambda + \lambda^{-1}) \\
i(\lambda +\lambda^{-1} )& 0
\end{pmatrix}.
\end{equation}
 It is easy to see that solutions of the pair of ordinary differential 
 equations $d F^{x} = F^{x} \eta^x$ and $d G^{y} = G^y \eta^y$ with 
 $F^x(x=0, \l) = G^y(y=0, \l) = \id$ can be computed as 
\eq{
 F^x(x) = \exp (x A) \;\;\mbox{and}\;\; 
 G^y(y) = \exp (-y A ).
}
 Consider the Birkhoff decomposition in Theorem \ref{thm:Birkhoff}
 as $(F^{x})^{-1} G^y = V_- V_+^{-1}$ 
 and set 
\eq{
 F = F^x V_- =  G^y V_+, 
}
 where $V_- \in \LSUNN$ and $V_+ \in \LSUP$.
 Then $F$ is the extended frame for a {\ps} surface of revolution:
 Let $\gamma$ be a translation of $(x, y)$ to $(x+p, y- p)$ with 
 some $p \in \R$. Then 
\eq{
 \gamma^{*} F^x:= F^x(x+p, y-p, \l) =  \exp (p A) F^x,  \\
 \gamma^* G^y:= G^y(x+p, y-p, \l) = \exp (p A)G^y.
}
 Thus  
 $(\gamma^* F^x)^{-1}  \gamma^* G^y = (F^x)^{-1} G^y$, and 
 $ V_-^{-1}\gamma^* V_-  =V_+^{-1}\gamma^* V_+$. The left-hand 
 and the right hand side take values in $\LSUNN$ and $\LSUP$, 
 respectively. Thus we have
\eq{
 \gamma^* V_{-}= V_{-}\;\;\mbox{and}\;\; \gamma^*V_+ = V_+.
}
 Therefore, we conclude that 
\eq{
 \gamma^* F =  \exp (p A ) F.
}
 Inserting $\gamma^* F$ into the Sym formula \eqref{eq:Symformula}, it is easy to see 
 that the resulting surface is a {\ps} surface of revolution.
 The sine-Gordon equation with coordinates $(\tilde x, \tilde y) 
 = (x + y, x-y)$ becomes 
\eq{
 u_{\tilde x \tilde x} - \sin u =0.
}
 Note that $u$ depends only on $\tilde x$ and the above equation 
 is an ordinary differential equation.
 It is well known that the above ordinary differential equation can be explicitly 
 solved by an elliptic function, see for example \cite{KS}.

 A discrete analogue of the above argument holds. Let $(\etn, \etm)$ 
 be a pair of matrices as follows:
\eq{
  \etn= \etm^{-1} = A_+ L A_-, \;\;\mbox{with}\;\; A_{\pm} = 
\frac{1}{\Delta_{\pm}} 
\begin{pmatrix} 
1 & \pm \frac{i}{2} q \l^{\pm1} \\
\pm \frac{i}{2} q \l^{\pm1}& 1
\end{pmatrix}\;\;
\mbox{and}\;\;L= \di (e^{ic}, e^{-ic}), 
}
 where $\Delta_{\pm} = \sqrt{1 + (q/2)^2 \l^{\pm 2}}$, $q \;(0< |q/2|<1)$ 
 and $c = \pi \ell^{-1} \;(\ell \in \Z_+)$ are some constants. 
 Here we can consider $ \xi_+ = A_+$, $P_{-}^r = LA_- $ and $P_{-}^{l} = \id$ 
 for $\etn$, and 
 $\xi_-= A_-^{-1}$, $P_{+}^r = L^{-1} A_+^{-1}$ and $P_+^l = \id$ for $\etm$, 
 respectively.
 Thus $(\etn, \etm)$ is a pair of discrete generalized potentials 
 in \eqref{eq:generalizedpot}.
 It is easy to see that solutions of the pair of ordinary difference
 equations $(\Fn)_1 = \Fn \etn$ and $(\Gm)_2 = \Gm \etm$ with 
 $\Fn(n=0) = \Gm(m=0) = \id$ as 
\eq{
 \Fn (n) = (A_+ L A_-)^n \;\;\mbox{and}\;\; 
 \Gm (m) =(A_+ L A_-)^{-m}.
}
 Here we denote $X^{-m}= (X^{-1})^m$ and $X^0 = \id$ for a matrix $X$.
 Since $P_{-}^r|_{\l = \infty} = LA_-|_{\l = \infty}= \di (e^{ic}, e^{-ic})$, 
 we can define the matrix $D$ as in Theorem \ref{thm:converseconstruction2}.
 Consider the Birkhoff decomposition in Theorem \ref{thm:Birkhoff}:
\eq{
 (\Fn D)^{-1} \Gm = V_- V_+^{-1},
}
 where $V_- \in \LSUNN, V_+ \in \LSUP$, and set $F = \Fn D V_- =  \Gm V_+$.
 Let $\gamma$ be a translation of $(n, m) \in \Z^2$ 
 to $(n+p, m- p) \in \Z^2$ with some $p \in \Z$. 
 Then $\gamma^{*} \Fn = (A_+ L A_-)^p \Fn$,
 $\gamma^* \Gm = (A_+ L A_-)^p \Gm$ and $\gamma^* $. Thus  
 $(\gamma^* \Fn D)^{-1}  \gamma^* \Gm = (\Fn D)^{-1} \Gm$, and 
 we have
\eq{
 \gamma^* V_{-}= V_{-}\;\;\mbox{and}\;\; \gamma^*V_+ = V_+.
}
 Therefore, we conclude that 
\eq{
 \gamma^* F =  (A_+ L A_-)^p F \gamma^* D.
}
 Note that $\gamma^* D \in {\rm U}_1$. We now set $M = (A_+ L A_-)^{p}$ and 
 insert $\gamma^* F$ into the Sym formula \eqref{eq:Symformula}:
\eq{
\gamma^* f = M f M^{-1} + \l  (\partial_{\l} M) M^{-1}.
} 
 Then a direct computation shows that $M f M^{-1}|_{\l=1}$ represents a 
 rotation around an axis $L= (1, 0, -3/4)$ and an angle $2 p \ell$. Moreover 
 an another direct computation shows that $\l  (\partial_{\l} M) M^{-1}|_{\l=1}$ 
 perpendicular to the axis $L$. Thus $\gamma^* f|_{\l=1}$ represents a rotation 
 around the axis $L$ through a point $p_0$. In fact $p_0$ can be explicitly computed as 
 $p_0 = -1/2\cot \ell \;(\ell \in \Z_+)$. Therefore the resulting surface is a discrete {\ps} 
 surface of revolution with period $\ell \in \Z_+$.
 The discrete sine-Gordon equation in \eqref{eq:danothersine-Gordon}
 becomes 
\begin{equation}\label{eq:DrotsinG}
 Q_{1} Q_{\bar 1} = \left(\frac{Q - k}{1-k Q}\right)^2, 
\end{equation}
 where $k =q^2/4$.
\begin{Remark}
 It is known that  Equation \eqref{eq:DrotsinG} is a special case of the ${\rm dP_{III}}$ equation  discussed in \cite{RGH}.
\end{Remark}\subsection{Discrete Amsler surface and discrete Painlev\'e I\!I\!I 
 equation}\label{sbsc:Painleve}
 We recall a construction of smooth Amsler surfaces via the nonlinear 
 d'Alembert formula.
 Let $\xi_{\pm}$ be a pair of normalized potentials as follows:
\eq{
 \xi_{+} = 
 \l Adx, \;\;  
 \xi_{-} = 
 - \l^{-1} Ady
 \;\;\mbox{with}\;\;
 A =
\begin{pmatrix} 
 0 & i \\
 i & 0
\end{pmatrix}.
}
 Let $F_{+}$ and $G_-$ be a pair of solutions of ordinary differential 
 equations $d F_+ = F_+ \xi_+$ and $d G_- = G_- \xi_-$ with 
 initial conditions 
 $F_+(x =0) = \di (e^{i s}, e^{-i s})$ and
 $G_-(y =0) = \di (e^{i \ell}, e^{-i \ell})$ with $s \neq \ell$.
 It is easy to compute the solutions as 
\eq{
 F_+(x, \l) = \di(e^{i s}, e^{-i s})\exp(x \l A), \;\;
 G_-(y, \l) = \di(e^{i \ell}, e^{-i \ell})\exp(- y \l^{-1} A).
}
 Consider the Birkhoff decomposition in Theorem \ref{thm:Birkhoff}:
\eq{
 F_+^{-1} G_- = V_- V_+^{-1},
} 
 where $V_- \in \LSUNN$ and $V_+ \in \LSUP$, and set $F = F_+ V_- = G_- V_+$.
 Let $\gamma$ be a transformation of $(x, y, \l)$ to 
 $(x p, yp^{-1}, \l p^{-1})$ with $p  \in \R^{\times}$.
 Then it is easy to see that 
\eq{
 \gamma^* F_+ = F_+ \;\;\mbox{and} \;\;
 \gamma^* G_- = G_-.
}
 Thus $\gamma^* (V_- V_+^{-1}) = F_+^{-1} G_-= V_- V_+^{-1}$, that 
 is,
\eq{
 V_-(x p, y p^{-1}, \l p^{-1})
 V_+(x p, y p^{-1}, \l p^{-1})^{-1} =
 V_-(x, y, \l) V_+(x, y, \l)^{-1}.
}
 Then $V_-(x p, y p^{-1}, \l p^{-1})$ and $V_+(x p, y p^{-1}, \l p^{-1})$
 clearly take values in $\LSUNN$ and $\LSUP$, respectively, and thus 
 $\gamma^* V_- = V_-$ and $\gamma^* V_+ = V_+$.  Finally, 
 we have
\eq{
 \gamma^* F = F.
}
 This implies that $F$ depends only on $r = xy$ and $\l$, and the 
 resulting surface is a Amsler surface. The sine-Gordon equation 
 becomes the so-called Painlev\'e I\!I\!I equation in trigonometric form 
 with fixed parameters \cite[p.~443]{Fokas}:
\eq{
 \tilde u^{\prime \prime}(r) + \frac{1}{r} \tilde u^{\prime}(r) + \sin \tilde u(r)=0,
}
 where $\tilde u(r) = u(x, y)$ and $r =2 \sqrt{- x y}$.
 We note that if the initial condition $F_+(x=0) = G_-(y=0)$, that is $s =\ell$, then 
 the resulting surface is a line.

 Moreover, the Amsler surface has an additional reflection symmetry 
 as follows: We first define a transformation $\delta$ as 
\eq{
 \delta^* L(x, y, \l)= \overline{L(y, x, \bar \l^{-1})}.
}
 It is easy to see that  $\delta^* F_+ = C_0 G_-$ and $\delta^* G_- = C_0 F_+$with 
 $C_0 = \di (e^{-i (\ell +s)}, e^{i(\ell + s)})$.
 Then a direct computation shows that 
 $(\delta^* (F_+^{-1} G_-))^{-1} 
 =  (\delta^* G_-)^{-1} \delta^* F_+ = F_+^{-1} G_-$.
 On the one hand, 
 from the decomposition $F_+^{-1} G_- = V_- V_+^{-1}$, we have $V_-V_+^{-1} 
 =  \delta^* V_+ (\delta^* V_-)^{-1}$: 
\eq{
 \delta^* V_- =V_+ d \;\;\mbox{and}\;\;  \delta^* V_+ = V_- d, 
}
 where $d^{-1} = V_+(\l =0) \in {\rm U}_1$. Finally using $\delta^* F_+ = 
 C_0 G_-$, we have 
\eq{
 \delta^* F =\delta^* (F_+ V_-) = C_0 G_- V_+ d = C_0 F d.
}
 Inserting this symmetry into the Sym formula, we have 
\eq{
 \delta^* f = \overline{f(y, x, \bar \l^{-1})}= - C_0 f(x, y, \l) C_0^{-1}.
}
 Thus, at $\l =1$, this implies 
\eq{
 f(y, x) = -C_0^{-1} \overline{f(x, y)}  C_0.
}
 This represents a reflection through a plane.
 Thus we have
\eq{
 u(y, x) = u(x, y) \;\;(\operatorname{mod} 2 \pi).
}

 We now give a discrete analogue of the above argument as follows.
 Let $\xi_{\pm}$ be a pair of discrete normalized potentials
\begin{equation}\label{eq:Amslerpot}
 \xi_+ =A_+, \;\; \xi_- =A_{-}\;\;\mbox{with}\;\;
 A_{\pm} = \frac{1}{\Delta_{\pm}}
\begin{pmatrix}
 1 & \pm \frac{i}{2} q \l^{\pm 1}  \\ 
 \pm \frac{i}{2} q \l^{\pm 1} & 1
\end{pmatrix},
\end{equation}
 where $\Delta_{\pm} = \sqrt{1+(q/2)^2 \l^{\pm 2}}$ and 
 $q \;(0< |q/2|<1)$ is some real constant.
 Note that $\alpha(n)= \beta(m) =0$ for $\xi_{\pm}$ in \eqref{eq:discretepotentials},
 thus the matrix $D$ in Theorem \ref{thm:converseconstruction} is identity matrix.
 It is easy to see that solutions of the pair of ordinary difference
 equations $(F_+)_1 = F_+ \xi_+$ and $(G_-)_2 = G_- \xi_-$ with 
 initial conditions 
 $F_+(n=0) = \di (e^{is}, e^{-is})$ and $G_-(m=0) =  
 \di (e^{i \ell}, e^{-i\ell}),\;(\ell \neq s)$ as 
 \begin{equation}\label{eq:FplusGminus}
 F_+(n) =\di (e^{is}, e^{-is})A_+^n \;\;\mbox{and}\;\; 
 G_-(m) = \di (e^{i\ell}, e^{-i\ell}) A_-^m.
 \end{equation}
 Here we use notation $A_{\pm}^{0} = \id$.
 Then a straightforward computation shows 
\begin{eqnarray*}
 F_+^{-1} G_- 
 & =& A_+^{-n}
 \di (e^{i (\ell -s)}, e^{-i(\ell -s)}) A_-^{m} \\
 &=& \frac{1}{\Delta_+^n \Delta_-^m}
\begin{pmatrix}
 1 & -\frac{i}{2} q \l \\
- \frac{i}{2} q \l & 1
\end{pmatrix}^{n}
\begin{pmatrix}
 e^{i (\ell -s)} & 0\\ 
 0 &  e^{-i(\ell -s)}
\end{pmatrix}
\begin{pmatrix}
1 & -\frac{i}{2}q \l^{-1} \\
-\frac{i}{2}q \l^{-1} & 1
\end{pmatrix}^m.
\end{eqnarray*}
 Consider the Birkhoff decomposition in Theorem \ref{thm:Birkhoff}
 as $F_+^{-1} G_-= V_- V_+^{-1}$  and set 
\eq{
 F = F_+ V_- =  G_- V_+,
}
 where $V_- \in \LSUNN$ and $V_+ \in \LSUP$.
 Then the discrete extended frame $F$ has the special property as follows.
\begin{Theorem}
 The $\lambda$-derivative of the discrete extended frame $F$ has the following form$:$
 \begin{equation}\label{eq:discretelambdaeq}
\l F^{-1} \partial_{\lambda} F = \frac{1}{\Delta_{+}^2 \Delta_{-}^2} 
 \begin{pmatrix}
 a   &   b \l^{-1}  +c \l \\
 - \bar b \l^{-1} - \bar c \l & - a 
 \end{pmatrix},
\end{equation}
 where $a \in i \R$ and $b, c \in \C$ are functions of $(n, m) \in \Z^2$.
\end{Theorem}
\begin{proof}
 Since $F = F_+ V_- = G_- V_+$, we have 
 \begin{equation}\label{eq:lambdaeq}
 \l F^{-1} \partial_{\l} F  = V_-^{-1} F_+^{-1} (\l \partial_{\l} F_+) V_- + V_-^{-1} \l \partial_{\l} V_-
 = V_+^{-1} G_-^{-1} (\l \partial_{\l} G_-) V_+ + V_+^{-1} \l \partial_{\l} V_+.
 \end{equation}
 Using the explicit forms of $F_+$ and $G_-$ in \eqref{eq:FplusGminus}, we compute
\begin{equation}\label{eq:lderivativeofF}
\l F_+^{-1} \partial_{\l} F_+ = 
\frac{i n q \l}{2 \Delta_+^2}
\begin{pmatrix}
 0 & 1 \\ 1 & 0
\end{pmatrix}
\;\;\mbox{and}   \;\;
\l G_-^{-1} \partial_{\l} G_- = 
\frac{i m q \l^{-1}}{2 \Delta_-^2} 
\begin{pmatrix}
 0 & 1 \\ 1 & 0
\end{pmatrix}.
\end{equation}
 Therefore, multiplying $(\Delta_+ \Delta-)^2$ 
 on the both sides of \eqref{eq:lambdaeq}, 
 we have the equality
 \begin{align}
 & \frac{inq \l}{2}\Delta_-^2 V_-^{-1} 
 \begin{pmatrix} 
 0 & 1 \\ 1 & 0
 \end{pmatrix} 
 V_- + (\Delta_+ \Delta_-)^2V_-^{-1} \l \partial_{\l} V_- \label{eq:lambdaeqality}\\
& =  \frac{i m q \l^{-1}}{2} \Delta_+^2 V_+^{-1} 
 \begin{pmatrix} 
 0 & 1 \\ 1 & 0
 \end{pmatrix} 
 V_+ + (\Delta_+ \Delta_-)^2V_+^{-1} \l \partial_{\l} V_+.
\nonumber
\end{align}
 Moreover, since $V_- \in \LSUNN$ and $V_+ \in \LSUP$, we have 
\eq{
  \l V_-^{-1} \partial_{\l} V_{-} & = \l (\id + \l^{-1} \tilde V_{(-1)}+\l^{-2} \tilde V_{(-2)}+ \cdots) (-\l^{-2} \check V_{(-1)}-2\l^{-3} \check V_{(-2)}+ \cdots), \\
 & = - \l^{-1} \check V_{(-1)} + \l^{-2} * + \cdots,
}
 and 
\eq{
 \l V_+^{-1} \partial_{\l} V_{+} &= \l (\hat V_{(0)} + \l \hat V_{(1)}+\l^{2} \hat V_{(2)}+ \cdots)
 (V_{(1)}+ 2\l V_{(2)}+ \cdots),\\
 &= \l \hat V_{(0)} V_{(1)} + \lambda^2 * + \cdots.
}
 Here we set $V_{-} = \id + \l^{-1} \check V_{(-1)}+\l^{-2} \check  V_{(-2)}+ \cdots$ and 
 $V_{+} = V_{(0)}+ \l V_{(1)}+\l^{2} V_{(2)}+ \cdots$.
 Noting that $(\Delta_+ \Delta-)^2 =  q^2(\l^2 + \l^{-2})/4 + q^4/16  +1 $, we have 
\eq{
 (\Delta_+ \Delta_-)^2 F^{-1} \l \partial_{\l} F = \l^{-1} F_{(-1)} +  F_{(0)} + \l F_{(1)}.
}
 Since $F$ takes values in $\LSU$, $\l F^{-1} \partial_{\l} F$ takes values 
 in $\lsu$, and $F_{(-1)}$ and $F_{(1)}$ are off-diagonal matrices and 
 $F_{(0)}$ is a diagonal matrix. This completes the proof.
\end{proof}
 The {\it discrete Amsler surface} was defined in \cite{Hoffmann} using 
 the discrete sine-Gordon equation constrained by the equation 
 \cite[Equation (4.26)]{Hoffmann}:
\begin{equation}\label{eq:additionalconstraint}
 (m-n)(Q -Q_{\bar 1 \bar 2}) = (n+m) \left( \frac{k - Q_{\bar 1}}{1- k Q_{\bar 1}}
 - \frac{k - Q_{\bar 2}}{1- k Q_{\bar 2}}\right),
\end{equation}
 where $k = q^2/4$.
 This is equivalent with that the $\lambda$-derivative of $F$ has the form in 
 \eqref{eq:discretelambdaeq}, \cite[Theorem 4.3]{Hoffmann}.
 Thus, we have the following corollary.
\begin{Corollary}
 The surface constructed by the pair of discrete normalized potentials 
 in \eqref{eq:Amslerpot} is actually the discrete Amsler surface. 
 The discrete sine-Gordon equation for the discrete Amsler surface is given as 
 follows:
\begin{equation}\label{eq:dPIII}
(m+n)Q Q_{\bar 1 \bar 2} + (m-n) (Q_{\bar 1 \bar 2}-Q) \left(\frac{Q_{\bar 1}-k}{1- k Q_{\bar 1}}\right)
 = (m+n)\left(\frac{ Q_{\bar 1}-k}{1- k Q_{\bar 1}}\right)^2,
\end{equation}
 where $k =q^2/4$.
\end{Corollary}
\begin{proof}
 In  the discrete sine-Gordon equation given in \eqref{eq:danothersine-Gordon}, 
 we eliminate $Q_{\bar 2}$ by \eqref{eq:additionalconstraint}.
 Then the equation \eqref{eq:dPIII} follows.
\end{proof}
 We finally show 
 that the resulting discrete surface has an additional reflection symmetry and 
 contains two straightlines as follows.
 Therefore it is natural to call the resulting surface $f$ a discrete Amsler surface.
\begin{Theorem}
 The resulting discrete {\ps} surface $f$ given by the pair of normalized potential $(\xi_+, \xi_-)$ in \eqref{eq:Amslerpot} has a reflection symmetry$:$
\eq{
 f(m, n) = - C_0^{-1} \overline{f(n, m)} C_0,  \;\;u(n, m) = u(m, n) \;\;({\rm mod} 2 \pi),
}
 where $C_0 = \di (e^{-i (\ell + s)}, e^{i(\ell + s)})$ and $u$ is the solution 
 to the discrete sine-Gordon equation.
 Moreover $f$ contains two straightlines.
\end{Theorem}
\begin{proof}
 The argument for proof is similar 
 to the case of the smooth Amsler surface:
 Define a transformation $\delta$ on a loop
\eq{
 \delta^* L(n, m, \l)= \overline{L(m, n, \bar \l^{-1})}.
}
 It is easy to see that  $\delta^* F_+ = C_0 G_-$ and $\delta^* G_- = C_0 F_+$ with 
 $C_0= \di (e^{-i (\ell +s)}, e^{i(\ell + s)})$.
 Then a direct computation shows that 
 $(\delta^*(F_+^{-1} G_-))^{-1} 
 =  (\delta^*G_-)^{-1} \delta^*F_+ = F_+^{-1} G_-$.
 On the one hand, 
 from the decomposition $F_+^{-1} G_- = V_- V_+^{-1}$, we have
\eq{
 V_-V_+^{-1} =  \delta^*V_+ (\delta^*V_-)^{-1}, \;\;\mbox{that is,}\;\;
 \delta^*V_- =V_+ d, \;  \delta^*V_+ = V_- d, 
}
 where $d^{-1} = V_+(\l =0) \in {\rm U}_1$. Finally using $\delta^* F_+ = C_0 G_-$,
 we have 
\eq{
 \delta^*F =\delta^*(F_+ V_-) = C_0 G_- V_+ d = C_0 F d.
}
 Inserting this symmetry into the Sym formula, we have 
\eq{
 \delta^* f = \overline{f(y, x, \bar \l^{-1})}= - C_0 f(x, y, \l) C_0^{-1}.
}
 Thus, at $\l =1$, this implies 
\eq{
 f(m, n) = -C_0^{-1} \overline{f(n, m)}  C_0.
}
 This represents a reflection through a plane and 
 the angle function $u$ has the symmetry 
 $u(n, m) = u(m, n) \;\;(\operatorname{mod} 2 \pi)$.

 Let us consider the case $n =0$. Then 
\eq{
 F_+^{-1} G_-|_{n=0} =  \di (e^{i (\ell -s)}, e^{-i(\ell -s)}) A_-^{m}, 
}
 and the Birkhoff decomposition can be computed explicitly as 
\eq{
 V_-|_{n=0} = \di (e^{i (\ell -s)}, e^{-i(\ell -s)}) A_-^{m} 
 \di (e^{-i (\ell -s)}, e^{i(\ell -s)}), \;\;
 V_+|_{n=0} = \di (e^{-i (\ell -s)}, e^{i(\ell -s)}).
}
 Thus $F|_{n=0} = F_+ V_-|_{n=0} = G_- V_+|_{n=0} =  \di (e^{i\ell}, e^{-i\ell}) A_-^{m} 
 \di (e^{-i (\ell -s)}, e^{i(\ell -s)})$ and 
\eq{
 V|_{n=0} = F^{-1} F_2|_{n=0} = \ad (\di (e^{i (\ell -s)}, e^{-i(\ell -s)})) A_-.
}
 Therefore $u(0, m) = 4(\ell -s ) - u(0,0)$ if $m$ is odd, and 
 $u(0, m) = u(0,0)$ if $m$ is even. Using the $\l$-derivative of $G_-$ in 
 \eqref{eq:lderivativeofF}, we compute that 
\eq{
f(0, m) = \l (\partial_{\l} F) F^{-1}|_{n=0, \l =1} = 
\frac{2 i m q}{4+q^2}
\begin{pmatrix}
 0 & e^{2 i \ell} \\ e^{-2 i \ell} & 0 
\end{pmatrix}.
}
Thus $f(0, m)$ represents a line in the $(x,y)$-plane through origin.

 Similarly let us consider the case $m =0$. Then 
\eq{
 F_+^{-1} G_-|_{m=0} = A_+^{-n} \di (e^{i (\ell -s)}, e^{-i (\ell -s)}),
}
 and the Birkhoff decomposition can be computed as 
\eq{
 V_-|_{m=0} = \id, \;\;
 V_+|_{m=0} =\di (e^{i (s-\ell)}, e^{-i (s-\ell)}) A_+^{n} .
}
 Thus 
 $ F|_{m=0} = F_+ V_-|_{m=0} = G_- V_+ |_{m=0}=  \di (e^{i s}, e^{-i s}) A_+^n$ and 
\eq{
U|_{m=0} = F^{-1} F_1|_{m=0} = A_+.
}
 Therefore $u(n, 0) = u(0, 0)$.
 Using the $\l$-derivative of $F_+$ in 
 \eqref{eq:lderivativeofF}, we compute that 
\eq{
f(n, 0) = \l (\partial_{\l} F) F^{-1}|_{m=0, \l =1} = 
\frac{2 i n q}{4+ q^2}
\begin{pmatrix}
 0 & e^{2 i s}  \\ e^{-2 i s} & 0 
\end{pmatrix}.
}
 Thus $f(n, 0)$ also represents a line in the $(x,y)$-plane through origin.
 This completes the proof.
\end{proof}
\appendix
\section{Discrete flow of discrete curves and discrete pseudospherical surfaces}\label{sc:App}
 It is known that the sine-Gordon equation  $u_{s t} - \sin u=0$ 
 (or equivalently {\ps} surfaces) 
 has another geometric interpretation;
 the flow of a constant torsion curve $\gamma$
 in $\E^3$,  see \cite{CI, CI2}.  In \cite{IKMO}, a discrete analogue 
 of the flow of constant torsion curves was introduced. In this 
 appendix we show that the discrete flow of a discrete constant torsion curve
 can be identified with a discrete {\ps} surface.
\subsection{Pseudospherical surfaces and 
 the flows of constant torsion space curves}
 Let $\gamma$ be a constant torsion curve $\gamma$ in $\E^3$
 and denote by $\kappa$ and $\tau$ the curvature and the torsion 
 of $\gamma$, respectively. 
 We use the parameter $s$ for 
 the curve $\gamma$. Moreover, set 
\eq{
 \kappa = u_s\;\;\mbox{and} \;\;\lambda = -i \tau.
}
 Then the Frenet frame for the curve $\gamma$ taking values in $\SU$
 is described in the first equation in \eqref{eq:curveLaxpair}.
 We now consider a special flow given by the second equation in 
 \eqref{eq:curveLaxpair} with the flow parameter $t$, see \cite{CI2}:
\begin{equation}\label{eq:curveLaxpair}
 \tilde F_s = \tilde F \tilde U \;\;\mbox{and} \;\;
 \tilde F_t  = \tilde F \tilde V,
\end{equation}
 where 
\begin{equation}\label{eq:movingUVflow}
 \tilde U= 
\frac{1}{2}
 \begin{pmatrix}
 \lambda   &  -u_s\\
  u_s     & - \lambda 
 \end{pmatrix}, \;\;
 \tilde V= \frac{1}{2\lambda} 
 \begin{pmatrix}
 \cos u   &  - \sin  u \\
  -\sin u &  - \cos u
 \end{pmatrix}.
\end{equation}
 Then the compatibility condition of the above system is 
 the sine-Gordon equation:
\eq{
 u_{s t} -\sin u =0.
}
 Moreover, the constant torsion space curve $\gamma$ 
 with curvature $\kappa = u_s$ and torsion $\tau$
 can be represented by the following formula:
\begin{equation}\label{eq:constcurve}
 \gamma (s, t) = \left.
 i \frac{\partial \tilde F}{ \partial \lambda} 
 \tilde F^{-1}\right|_{\lambda = -i \tau}.
\end{equation}
 The $\lambda = -i \tau$ can be considered as a parameter of 
 the system \eqref{eq:curveLaxpair} and $\tilde F$ is an element 
 of the loop group, which is not in $\LSU$, since 
 the twisted condition is not satisfied.  
 We call $\tilde F$ the {\it extended frame} 
 of the curve $\gamma$.

 The extended frame $\tilde F$ and the formula for the curve $\gamma$
 in \eqref{eq:constcurve} is similar to the case 
 of {\ps} surfaces in Section \ref{subsc:pssurf}. In fact they are 
 gauge equivalent as follows, see also \cite[Section 2.1]{CI}.
\begin{Proposition}
 The extended frame $F$ of a {\ps} surface and the extended 
 frame $\tilde F$ of the flow of a constant 
 torsion curve are conjugate by some constant matrix $C$:
 \begin{equation}\label{eq:gauge}
 F(x, y, \lambda) =  C^{-1} \tilde F(x, y, i \lambda)C, 
 \;\;\mbox{with}\;\; 
 C = \frac{1}{\sqrt{2}}
 \begin{pmatrix}
  e^{\frac{i \pi}{4} }&  e^{\frac{i \pi}{4} }\\
 - e^{-\frac{i \pi}{4}} &  e^{-\frac{i \pi}{4}}
 \end{pmatrix}.
\end{equation}
 Moreover, the immersion $f$ in \eqref{eq:Symformula} 
 and the flow of constant torsion 
 curve $\gamma$ in \eqref{eq:constcurve}
 are related by rotation of the matrix $C:$
\eq{
 f^{\lambda}|_{\l = - \tau} = (-i \tau) C^{-1}  \gamma C.
}
\end{Proposition}
\begin{proof}
 It is straightforward to compute 
 $F_{x} = FU$ and $F_{y} = FV$ 
 in terms of $\tilde F$ as follows:
\eq{
 \tilde F_x =  \tilde F (C U C^{-1}) \;\;\mbox{and} \;\;
 \tilde F_y =  \tilde F (C V C^{-1}).
}
 Then it is easy to see that $ C U C^{-1} = \tilde U$ and 
 $CV C^{-1} = \tilde V$ with $\l$ replaced by $i \l$.
\end{proof}

\begin{Remark}
 $\tilde F$ is an element of some loop group, but it does 
 not satisfy the usual twisted condition, 
 that is, $\ad \sigma_3 g (\l) = g(-\l)$. However, it satisfies 
 another twisted condition 
 $\ad (\sigma_3C^{-1}) g (\l) = g(-\l)$,
 where $C$ is given in \eqref{eq:gauge}.
\end{Remark}

\subsection{Discrete flow of a discrete constant torsion 
space curve}
 From the result in \cite[Theorem 6.2]{IKMO}, we discuss a discrete flow of a discrete 
 constant torsion space curve.
 Let $\gamma$ be a discrete constant torsion 
 curve parametrized by $n\in \Z$ and denote the torsion by $\tau$.
 Then we evolve the curve $\gamma$ as follows:
 \begin{align*}
 \gamma_{2} = \gamma + \delta(\cos w T + \sin w N), & \;\;
 \delta = \frac{b}{1 + (b/2)^2\l^2}, \\
 \tan \frac{w_1 + k_1}{2} = \frac{b+ a}{b- a} \tan \frac{w}{2},&
\end{align*}
 where $T$ and $N$ are the tangent and the principal normal 
 vectors to $\gamma$, respectively, and $a$ and $c$ depend only on $n$ 
 and $m$, respectively\footnote{Note that in \cite{IKMO}, the identification of $\E^3$ and $\su$ is different from our paper and they use different notation; 
 the relations are given by $a(n) = a_n$ and $c(m) = 1/b_m$.}:
\eq{
 a = a(n), \;\; b = b(m).
}
 It is known that $k = k(n, m)$ and $w = w(n, m)$ can be rephrased by $u= u(n, m)$
 \cite[$(8.3)$ and $(6.22)$]{IKMO} as 
\eq{
k = \frac{u_1 -u_{\bar 1}}{2}, \;\;\mbox{and}\;\;
w = -\frac{u_2 +u_1}{2}.
}
 The discrete Frenet frame $\tilde \Phi = (T, N, B)$ taking values in ${\rm SO}_3$
 with the binormal vector $B$
 of the curve $\gamma$ satisfies \cite[$(6. 14)$ and $(6.16)$]{IKMO}
\eq{
 \tilde \Phi_1 = \tilde \Phi\tilde L, \;\;\;  \tilde \Phi_2 = \tilde \Phi \tilde M, \;\;
}
 with 
\eq{
 \tilde L = \tilde R^{(1)} (-\nu_1) \tilde R^{(3)} (\kappa_1), \;\;
 \tilde M= \tilde R^{(3)} (w) \tilde R^{(1)} (\mu) \tilde R^{(3)} (- w_1 - \kappa_1),  
}
 where $\nu_1 = 2 \arctan \frac{a \l}{2}$, $\mu = 2 \arctan \frac{2}{b \l}$
 and $\tilde R^{(j)}(x)$ denotes the rotation matrix around 
 the vector $e_j$ with angle $x$.
 Then there is a $2$ to $1$ corresponding from $\SU$ to ${\rm SO}_3$ such that
 the corresponding matrices $R^{(j)} \in \SU$ have the following forms:
\begin{align*}
& {\rm SO}_3 \ni \tilde R^{(1)}(x) \longleftrightarrow  R^{(1)}(x) =\begin{pmatrix} \cos \frac{x}{2}&  -i \sin \frac{x}{2}\\   -i \sin \frac{x}{2}&  \cos \frac{x}{2}\end{pmatrix} \in \SU, \\
& {\rm SO}_3 \ni \tilde R^{(3)}(x) \longleftrightarrow   R^{(3)}(x) =\di \left(e^{-i \frac{x}{2}}, e^{i  \frac{x}{2}}\right) 
 \in \SU.
\end{align*}
 Thus we translate the 
 Frenet frame $\tilde \Phi$ in ${\rm SO}_3$ into the corresponding 
 frame $\Phi$ in $\SU$ see for example 
 \cite[Section 2]{IKMO} for the explicit correspondence.
 From this correspondence, it is clear that the discrete Frenet frame $\Phi$ satisfies 
\begin{equation}\label{eq:franetframe}
 \Phi_1 = \Phi L, \;\;
 \Phi_2 = \Phi M, \;\;
\end{equation}
 where 
\eq{L = R^{(1)} (-\nu_1)R^{(3)} \left(\frac{u_{1 1} -u}{2}\right), \;\;
 M = R^{(3)} \left(-\frac{u_2+u_1}{2}\right) R^{(1)} (\mu) R^{(3)} \left(\frac{u_{12}+u}{2}\right),
}
 where $u_{1 1}= u(n+2, m)$.
 Then we have the following proposition. 
 It is already mentioned in \cite[Remark 8.3]{IKMO}, but 
 we give a proof for completeness.
\begin{Proposition}
 The discrete Frenet frame $\Phi$ in \eqref{eq:franetframe} and the discrete 
 extended frame $F$ in \eqref{eq:DLaxpair} are related by
\eq{
 F = \Phi G, \;\; G= \di \left(e^{i \frac{u_1-u}{4}}, e^{-i \frac{u_1 - u}{4}}
 \right).
}
 Moreover both $F$ and $\Phi$ define a same discrete {\ps} 
 surface. 
\end{Proposition}
\begin{proof}
 We first rephrase the matrices $U$ and $V$ in \eqref{eq:DmovingUV} as
\eq{
 U = \begin{pmatrix}  e^{- \frac{i}{2} (u_1 -u)} \cos \frac{p^{\l}}{2}  & i 
 \sin \frac{p^{\l}}{2} 
 \\ i \sin \frac{p^{\l}}{2} &  e^{\frac{i}{2}(u_1 -u)} \cos \frac{p^{\l}}{2}\end{pmatrix}, \;\; 
 V= \begin{pmatrix} \cos \frac{q^{\l}}{2} & -i e^{-\frac{i}{2} (u_2 + u)} \sin \frac{q^{\l}}{2} \\
 -i e^{\frac{i}{2}(u_2 + u)} \sin \frac{q^{\l}}{2}  & \cos \frac{q^{\l}}{2} \end{pmatrix},
}
  where $p^{\l} =2 \arccos \left(1/\sqrt{1 + (p/2)^2 \l^2} \right)$ and 
 $q^{\l} = 2\arccos \left(1/\sqrt{1 + (q/2)^2 \l^{-2}}\right)$.
 Let us compute the moving frame equations for $\Phi G$:
\eq{
 (\Phi G)^{-1} (\Phi G)_1 = G^{-1} L G_1, \;\;
 (\Phi G)^{-1} (\Phi G)_2 = G^{-1} M G_2. 
}
 Since $G$ can be rephrased by $G = R^{(3)}\left(-\frac{u_1 -u}{2}\right)$ and using 
 the property of the rotation matrix, we compute 
\eq{
 G^{-1} L G_1 = R^{(3)}(\tfrac{u_1 -u}{2}) R^{(1)}(-\nu_1) 
  R^{(3)}(\tfrac{u_1 -u}{2}), \;\;G^{-1} M G_2 = 
 R^{(3)}(-\tfrac{u_2 +u}{2}) R^{(1)}(\mu) 
  R^{(3)}(\tfrac{u_2 +u}{2}).
}
 Then a straightforward computation shows that 
 $G^{-1} L G_1 = U$ and $G^{-1} M G_2 =V$ under the 
 identification $a= p$ and $b  = 4/q$. Here we use the identity 
 $\arctan(x) = \arccos(1/\sqrt{1+x^2})$ for $x>0$.
 This completes the proof.
\end{proof}

\bibliographystyle{plain}
\def\cprime{$'$}

\end{document}